\documentclass{amsart}
\usepackage{amsfonts,amssymb,amscd,amsmath,enumerate,verbatim,calc}
\usepackage[all]{xy}

\newcommand{\CM}{Cohen-Macaulay}

\newcommand{\wrt}{with respect to}

\newcommand{\m}{\mathfrak{m} }
\newcommand{\Bcal}{\mathcal{B} }
\newcommand{\M}{\mathfrak{M} }

\newcommand{\R}{\mathcal{R} }
\newcommand{\F}{\mathcal{F} }
\newcommand{\G}{\mathcal{G} }
\newcommand{\Z}{\mathbb{Z} }

\newcommand{\rt}{\rightarrow}
\newcommand{\xar}{\longrightarrow}
\newcommand{\ov}{\overline}

\newcommand{\wh}{\widehat }
\newcommand{\wt}{\widetilde }

\newcommand{\reg}{\operatorname{reg}}

\newcommand{\depth}{\operatorname{depth}}
\newcommand{\e}{\operatorname{end}}

\newcommand{\red}{\operatorname{red}}

\newcommand{\Syz}{\operatorname{Syz}}

\theoremstyle{plain}

\newtheorem{theorem}{Theorem}[section]
\newtheorem{corollary}[theorem]{Corollary}
\newtheorem{lemma}[theorem]{Lemma}
\newtheorem{proposition}[theorem]{Proposition}

\theoremstyle{definition}

\newtheorem{s}[theorem]{}

\newtheorem{example}[theorem]{Example}

\theoremstyle{remark}

\begin{document}

\title{On generalized  Narita ideals}
\author{Tony~J.~Puthenpurakal}
\date{\today}
\address{Department of Mathematics, IIT Bombay, Powai, Mumbai 400 076, India}

\email{tputhen@gmail.com}


 \begin{abstract}
Let $(A,\m)$ be a \CM \ local ring of dimension $d \geq 2$. An $\m$-primary ideal $I$ is said to be a generalized  Narita ideal if $e_i^I(A) = 0$ for $2 \leq i \leq d$. If $I$ is a generalized  Narita ideal and $M$ is a maximal \CM \ $A$-module then we show $e_i^I(M) = 0$ for $2 \leq i \leq d$. We also have $G_I(M)$ is generalized  \CM.  Furthermore we show that there exists $c_I$ (depending only on $A$ and $I$) such that   $\reg G_I(M) \leq c_I$.
\end{abstract}
 \maketitle
\section{introduction}
Let $(A,\m)$ be a \CM \ local ring of dimension $d$ and let $I$ be an $\m$-primary ideal. Let $G_I(A) = \bigoplus_{n \geq 0}I^n/I^{n+1}$ be the associated graded ring of $I$. Let $M$ be a MCM
(maximal \CM) $A$-module. Let $G_I(M) = \bigoplus_{n \geq 0}I^nM/I^{n+1}M$ be the associated graded module of $M$ with respect to $I$ (considered as a $G_I(A)$-module).

We assume for the time being that $A$ is complete. Then usually  $A$ is \emph{not} of finite representation type. In particular there are infinitely many non-isomorphic indecomposable MCM $A$-modules. In this paper we are interested in properties of $I$ which impose conditions on $G_I(M)$\textit{ for all} MCM $A$-modules.
One such property is when $I$ has minimal multiplicity. Then it is easy to check that every MCM $A$-module $M$ has minimal multiplicity with respect to $I$.
In this paper we give another class of ideals for which we can conclude good properties of all $G_I(M)$ from a condition on $I$.

Narita showed that if $A$ is \CM \ of dimension two and $e_2^I(A) = 0$ then reduction number of $I^n$ is one for all $n \gg 0$. In particular reduction number of $I^n$ with respect to MCM $A$-module $M$ is one for all $n \gg 0$, \cite{N}. So $e_2^I(M) = 0$ and  $G_{I^n}(M)$ is \CM \ for all $n \gg 0$. In particular $G_I(M)$ is generalized \CM.
Narita's result does not directly generalize to higher dimensions. However in dimension two Narita's result $e_2^I(M) = 0$ is equivalent to the fact that $\wt{G}_I(M)$, the associated graded module of the Ratliff-Rush filtration of $M$ with respect to $I$, is \CM \ with minimal multiplicity. In this form Narita's result was generalized to higher dimensions. In \cite[6.2]{Pu6} we proved that if $M$ is \CM \ of dimension $r \geq 2$ then $e_i^I(M) = 0$ for $2\leq i \leq r$ if and only if $\wt{G}_I(M)$, the associated graded module of the Ratliff-Rush filtration of $M$ with respect to $I$, is \CM \ with minimal multiplicity.
Motivated by this result we define an $\m$-primary ideal in $A$ to be a \emph{generalized Narita ideal} if $e_i^I(A) =0$ for $2\leq i \leq d$. For such ideals we prove the following result

\begin{theorem}\label{main}
Let $(A,\m)$ be a \CM \ local ring of dimension $d \geq 2$. Let $I$ be a generalized  Narita ideal. Let  $M$ be a MCM $A$-module. Then
 \begin{enumerate}[\rm (1)]
   \item $e_i^I(M) = 0$ for $2 \leq i \leq d$.
 %  \item $H^i(L^I(M)) = 0$ for $1\leq i \leq d-1$.
   \item $\wt{G}_I(M)$ is \CM \ with minimal multiplicity.
   \item $G_{I^n}(M)$ is \CM \ for all $n \gg 0$.
   \item $G_I(M)$ is generalized \CM.
 %\item Let $H = A^r \rt M \rt 0$ be a projective cover. Then the map $H^0_\M(L^I(H)) \rt H^0_\M(L^I(M))$ is surjective.
 \end{enumerate}
 \end{theorem}

 Next we show
 \begin{theorem}\label{reg}
 Let $(A,\m)$ be a \CM \ local ring of dimension $d \geq 2$. Let $I$ be an $\m$-primary generalized Narita ideal. Let $M$ be a MCM $A$-module. Then there exists $c_I$ (depending only on $A$ and $I$) such that
  $\reg G_I(M) \leq c_I$.
 \end{theorem}

An easy corollary to Theorem \ref{main}
is
\begin{corollary}\label{e2r}
Let $(A,\m)$ be a \CM \ local ring of dimension $d \geq 2$. Let $I$ be an $\m$-primary ideal with $ e_i^I(A) = 0$ for $2 \leq i \leq r$ where $r \leq d$. Let  $M$ be a MCM $A$-module. Then
$e_i^I(M) = 0 $ for $2\leq i \leq r$.
\end{corollary}

Next we relax the condition of generalized Narita ideals by one index, i.e., we consider the case when $e_2^I(A) = e_3^I(A) = \cdots = e_{d -1}^I(A) = 0$. In this case we prove
\begin{theorem}\label{almost}
Let $(A,\m)$ be a \CM \ local ring of dimension $d \geq 3$. Let $I$ be an $\m$-primary ideal with $ e_i^I(A) = 0$ for $2 \leq i \leq d -1$. Let  $M$ be a MCM $A$-module. Then $(-1)^de_d^I(M) \geq 0$.

(I) The following conditions are equivalent:
\begin{enumerate}[\rm (i)]
  \item  $e_d^I(M) = 0$.
  \item $G_I(M)$ is generalized \CM.
  \item $G_{I^n}(M)$ is \CM \ for all $n \gg 0$.
\end{enumerate}

(II)  The following conditions are equivalent:
\begin{enumerate}[\rm (a)]
  \item  $e_d^I(M) \neq  0$.
  \item $ \depth G_{I^n}(M) = 1$  for all $n \gg 0$.
\end{enumerate}
\end{theorem}
\section{Preliminaries}
In this paper all rings considered are Noetherian and all modules (unless stated otherwise) are assumed to be finitely generated.
Let $A$ be a local ring, $I$ an ideal in $A$ and let $N$ be an $A$-module. Then set
 $\ell(N)$ to be length of $N$ and $\mu(N)$ the number of minimal generators of $N$.

\s Let $A$ be local and let $I$ be an $\m$-primary ideal. Let $M$ be an  $A$-module.
Recall an $I$-\textit{filtration} $\F = {\{\F_n\}}_{n \geq 0}$ on $M$ is a
collection of submodules of $M$ with the properties
\begin{enumerate}
\item
$\F_n \supseteq \F_{n+1}$ for all $n \geq 0$,
\item
$\F_0 = M$,
\item
$I \F_n \subseteq \F_{n+1}$ for all $n \geq 0$.
\end{enumerate}
If
$I \F_n = \F_{n+1}$ for all $n \gg 0$ then we say $\F$ is $I$-\emph{stable}.
\emph{We do NOT assume $\F_1 \neq M$}
\s Let $\F$ be an $I$-stable filtration on  $M$. Then the function $H^{(1)}(\F, n) = \ell(M/\F_{n+1})$ is of polynomial type of degree $r = \dim M$, i.e., there exists a polynomial $P_\F^{(1)}(X) \in \mathbb{Q}[X]$
such that $P_\F^{(1)}(n) = \ell(M/\F_{n+1}) $ for all $n \gg 0$. We write
\[
P_\F^{(1)}(X) = e_0^\F(M)\binom{X+d}{d} - e_1^\F(M) \binom{X+ d - 1}{d-1} + \cdots + (-1)^re_r^\F(M).
\]
The integers $e_i^\F(M)$ are called Hilbert coefficients of $M$ with respect to  $\F$. The integer $e_0^\F(M)$ is called the multiplicity of $M$  \wrt \  $\F$  and is always positive if $M \neq 0$.

\s Let $\F$ be an $I$-stable filtration on  $M$. The function $H(\F, n) = \ell(\F_n/\F_{n+1})$ is called the Hilbert function of $M$ with respect to $\F$. Let $G_\F(M) = \bigoplus_{n \geq 0} \F_n/\F_{n+1}$ be the associated graded module of $M$ \wrt \ $\F$.
It is well-known that $G_\F(M)$ is a finitely generated $G_I(A)$ module of dimension $r = \dim M$. It follows that
\[
\sum_{n \geq 0}H(\F, n)z^n = \frac{h_\F(z)}{(1-z)^r}.
\]
where $h_\F(z) \in \Z[z] $ and $h_\F(1) \neq 0$. It is easily verified that
$e_i^\F(M) = h_\F^{(i)}(1)/i!$ for $0 \leq i \leq r$. It is convenient to set $e_i^\F(M) = h_\F^{(i)}(1)/i!$ for all $i$.

\s An element $x \in I$ is said to $M$-superficial \wrt \ $\F$ if there exists $c$ such that $(\F_{n+1} \colon x)\cap \F_c = \F_n$ for all $n \gg 0$. It is well-known that superficial elements exist if the residue field $k = A/\m$ of $A$ is infinite. If $\depth M > 0$ it is easily verified that $x$ is $M$-regular. Furthermore $(\F_{n+1} \colon x) = \F_n$ for all $n \gg 0$.

\s Let $x$ be $M$-superficial \wrt \ $\F$. Consider the quotient filtration
$\ov{\F} = \{ \ov{\F_n} = (\F_n + xM)/xM \}_{n \geq 0}$. If $x$ is also $M$-regular then
$e_i^{\ov{\F}}(\ov{M}) = e_i^\F(M)$ for $i = 0, \ldots, \dim M -1$ (see proof in \cite[Corollary 10]{Pu1} for the adic case which generalizes).

\s If the residue field of $A$ is finite then we go the extension $A' = A[X]_{\m A[X]}$.
If $E$ is an $A$-module then set $E' = E \otimes_A A'$. If $\F = \{ \F_n \}$ is an $I$-stable filtration on $M$ then $\F' = \{ \F_n' \}$ is an $I'$-stable filtration on $M'$.
Furthermore we have $\depth G_\F(M) =  \depth G_{\F'}(M')$ and $e_i^\F(M) = e_i^{\F'}(M')$
for all $i \geq 0$.

\s Let $x \neq 0$. Say  $x \in I^n \setminus I^{n+1}$. The image of $x$ in $I^n/I^{n+1}$ is denoted by $x^*$ and is called the initial form of $x$.
 \section{Ratliff-Rush filtration of an $I$-stable filtration}
 In this section $(A,\m)$ is a local ring with infinite residue field, $I$ is an $\m$-primary ideal and $N$ is an $A$-module with $\depth N > 0$. Assume $\dim N = t$. Let $\F = \{\F_n \}_{n \geq 0}$ be an $I$-stable filtration on $N$. In this section we define Ratliff-Rush filtration of $\F$. When $N = A$ this has been done earlier in \cite{H} and \cite{B}.

 \s We have an ascending chain of submodules of $N$,
 $$ \F_n \subseteq (\F_{n+1} \colon I) \subseteq (\F_{n+2} \colon I^2) \subseteq \cdots \subseteq (\F_{n+r} \colon I^r) \subseteq \cdots. $$
 Set
 $$\wt{\F_n} = \bigcup_{r \geq 1}(\F_{n+r} \colon I^r).$$
 The main result of this section is
 \begin{proposition}\label{basic-rr}
 (with hypotheses as above) We have
 \begin{enumerate}[\rm (1)]
 \item $\wt{\F_n} = \F_n$ for all $n \gg 0$.
   \item  $\wt{\F} = \{ \wt{\F_n} \}_{n \geq 0}$ is an $I$-stable filtration on $N$.
   \item $e_i^{\wt{\F}}(N) = e_i^\F(N)$ for $i = 0, \ldots, t$.
   \item If $x$ is $N$-superficial with respect to $\F$ then $(\wt{\F_{n+1}} \colon x) = \wt{\F_n}$ for $n \geq 0$. In particular $x$ is $G_{\wt{\F}}(N)$-regular.
 \end{enumerate}
 \end{proposition}
 \begin{proof}
   Let $x$ be $N$-superficial with respect to $\F$. In particular there exists $n_0$ such that $(\F_{n+1} \colon x) = \F_n$ for $n \geq n_0$.

   (1) Fix $n \geq n_0$. Let $m \in \wt{\F}_n$. Then $I^jm \subseteq \F_{n+j}$ for some $j \geq 1$. So $x^j m \in \F_{n+j}$. It follows that $m \in (\F_{n+j} \colon x^j) = \F_n$. The result follows.

   (2) Let $p \in \wt{\F_{n+1}}$. Then $I^jp \subseteq \F_{n+j+1}$ for all $j \gg 0$. We also have $\F_{n+j+1} \subseteq \F_{n+j}$. So $p \in \wt{\F_n}$.

   Let $m \in \wt{\F_n}$ and let $\alpha \in I^j$. We have $I^im \subseteq \F_{i + n}$  for some $i$. So $I^j \alpha m \subseteq \F_{i + j + n}$. It follows that $\alpha m \in \wt{\F}_{j + n}$.
    So
   $\wt{\F}$ is an $I$-filtration.

    By (1) and as $\F$ is $I$-stable filtration on $N$ it follows that $\wt{\F}$ is $I$-stable filtration on $N$.

   (3) This follows from (1).

   (4) By (2) it follows that $\wt{\F_n} \subseteq (\wt{\F_{n+1}} \colon x) $ for all $n \geq 0$. Let $m \in (\wt{\F_{n+1}} \colon x)$.
   So $I^jxm \in \F_{n+1+j} $ for all $j \gg 0$. As $x$ is $N$-superficial with respect to $\F$ it follows that $I^jm \in \F_{n+j}$. So $m \in \wt{\F_n}$.
 \end{proof}

 \section{The $L$-construction for filtration's}\label{Lprop}
In \cite{Pu5} we introduced a new technique to investigate problems relating to associated graded modules. This was done for the adic-case. The technique can be extended to filtrations and analogous properties can be proved (with the same proofs).
In this section we collect all the relevant results which we proved in \cite{Pu5} in the adic-case. Throughout this section
$(A,\m)$ is a  local ring with infinite residue field, $M$ is a \emph{\CM }\ module of dimension
$r \geq 1$ and $I$ is an $\m$-primary ideal. Furthermore $\F = \{ \F_n \}_{n \geq 0}$ is an $I$-stable filtration on $M$ with $\F_0 = M$. We \emph{do not} insist that $\F_1 \neq M$.
Set $L^\F(M) = \bigoplus_{n\geq 0}M/\F_{n+1}$.
\s \label{mod-struc} Set $\R = A[I t]$;  the Rees Algebra of $I$.
Let $\R(\F, M) = \bigoplus_{n \geq 0}M_n$ be the Rees module of $M$ with respect to $\F$.
We note that  $M[t]/\R(\F, M) = L^\F(M)(-1)$.
Thus $L^\F(M)$ is a $\R$-module. Note $L^\F(M)$ is NOT a finitely generated $\R$-module.

\s Set $\M = \m\oplus \R_+$. Let $H^{i}(-) = H^{i}_{\M}(-)$ denote the $i^{th}$-local cohomology functor \wrt \ $\M$. Recall a graded $\R$-module $E$ is said to be
*-Artinian if
every descending chain of graded submodules of $E$ terminates. For example if $U$ is a finitely generated $\R$-module then $H^{i}(U)$ is *-Artinian for all
$i \geq 0$.

\s \label{Artin}
For   $0 \leq i \leq  r - 1$
\begin{enumerate}[\rm (a)]
\item
$H^{i}(L^\F(M))$ are  *-Artinian $\R$-modules; see \cite[4.4]{Pu5}.
\item
$H^{i}(L^\F(M))_n = 0$ for all $n \gg 0$; see \cite[1.10]{Pu5}.
\item
 $H^{i}(L^\F(M))_n$  has finite length
for all $n \in \mathbb{Z}$; see \cite[6.4]{Pu5}.
\item
$\ell(H^{i}(L^\F(M))_n)$  coincides with a polynomial for all $n \ll 0$; see \cite[6.4]{Pu5}.
\end{enumerate}

\s \label{l-rr}Let $\F$ be an $I$-stable filtration on $N$ and let $\wt{\F}$ be its Ratliff-Rush filtration. We consider the short exact sequence
 \[
 0 \rt \bigoplus_{n \geq 0} \frac{\wt{\F_{n+1}}}{\F_{n + 1}} \rt L^\F(N) \rt L^{\wt{\F}}(N) \rt 0.
 \]
 We have $H^0(L^{\wt{\F}}(N)) = 0$ by \ref{basic-rr}(4). Also $\bigoplus_{n \geq 0} \frac{\wt{\F_{n+1}}}{\F_{n + 1}} $ has finite length by \ref{basic-rr}(1). It follows that
 $$H^0(L^\F(N)) = \bigoplus_{n \geq 0} \frac{\wt{\F_{n+1}}}{\F_{n + 1}}.$$
 and
 $$H^i(L^\F(N)) = H^i(L^{\wt{\F}}(N)) \quad \text{for $i \geq 1$}.$$

 \s \label{I-FES} The natural maps $0\rt \F_n/\F_{n+1} \rt M/ \F_{n+1}\rt M/\F_n \rt 0 $ induce an exact
sequence of $\R$-modules
\begin{equation}
\label{dag}
0 \xar G_{\F}(M) \xar L^\F(M) \xrightarrow{\Pi} L^\F(M)(-1) \xar 0.
\end{equation}
We call (\ref{dag}) \emph{the first fundamental exact sequence}.  We use (\ref{dag}) also to relate the local cohomology of $G_\F(M)$ and $L^\F(M)$.

\s \label{II-FES} Let $x$ be  $M$-superficial \wrt \ $\F$ and set  $N = M/xM$ and $u =xt \in \R_1$. Let $\ov{\F} = \{ (\F_n + xM)/xM \}_{n \geq 0}$ be the quotient filtration on $N$. Notice $$L^\F(M)/u L^\F(M) = L^{\ov{\F}}(N)$$
For each $n \geq 1$ we have the following exact sequence of $A$-modules:
\begin{align*}
0 \xar \frac{\F_{n+1}\colon x}{\F_n} \xar \frac{M}{\F_n} &\xrightarrow{\psi_n} \frac{M}{\F_{n+1}} \xar \frac{N}{\ov{\F}_{n+1}} \xar 0, \\
\text{where} \quad \psi_n(m + M_n) &= xm + \F_{n+1}.
\end{align*}
This sequence induces the following  exact sequence of $\R$-modules:
\begin{equation}
\label{dagg}
0 \xar \Bcal^{\F}(x,M) \xar L^{\F}(M)(-1)\xrightarrow{\Psi_u} L^{\F}(M) \xrightarrow{\rho^x}  L^{\ov{\F}}(N)\xar 0,
\end{equation}
where $\Psi_u$ is left multiplication by $u$ and
\[
\Bcal^{\F}(x,M) = \bigoplus_{n \geq 0}\frac{(\F_{n+1}\colon_M x)}{\F_n}.
\]
We call (\ref{dagg}) the \emph{second fundamental exact sequence. }
\s \label{long-mod} Notice  $\ell\left(\Bcal^{\F}(x,M) \right) < \infty$. A standard trick yields the following long exact sequence connecting
the local cohomology of $L^\F(M)$ and
$L^{\ov{\F}}(N)$:
\begin{equation}
\label{longH}
\begin{split}
0 \xar \Bcal^{\F}(x,M) &\xar H^{0}(L^{\F}(M))(-1) \xar H^{0}(L^{\F}(M)) \xar H^{0}(L^{\ov{\F}}(N)) \\
                  &\xar H^{1}(L^{\F}(M))(-1) \xar H^{1}(L^{\F}(M)) \xar H^{1}(L^{\ov{\F}}(N)) \\
                 & \cdots  \\
               \end{split}
\end{equation}

\s \label{Artin-vanish} We will use the following well-known result regarding *-Artinian modules quite often:

Let $V$ be a *-Artinian $\R$-module.
\begin{enumerate}[\rm (a)]
\item
$V_n = 0$ for all $n \gg 0$.
\item
If $\psi \colon V(-1) \rt V$ is a monomorphism then $V = 0$.
\item
If $\phi \colon V \rt V(-1)$ is a monomorphism then $V = 0$.
\end{enumerate}

\s \label{l-induct}(1) For $c \leq r -1$ we have $H^i(G_\F(M)) = 0$ for $i = 0, \ldots, c$ if and only if $H^i(L^\F(M)) = 0$ for $i = 0, \ldots, c$, see \ref{I-FES} and \ref{Artin-vanish}.

(2) Let $x \in I$ be $M$-superficial \wrt \ $\F$.  Let $\ov{\F}$ be the quotient filtration on $\ov{M} = M/xM$. Then for $a \geq 1$ we have that if $H^i(L^{\ov{\F}}(M/xM)) = 0$ for $a \leq i \leq s \leq r-2$ then
$H^i(L^\F(M)) = 0$ for $a + 1 \leq i \leq s+1$, see \ref{II-FES}  and \ref{Artin-vanish}.

 \section{A Lemma when $d = 2$}
 In this section we prove
 \begin{lemma}\label{e2} Let $(A, \m)$ be a \CM \ local ring and let $I$ be an $\m$-primary ideal. Let $N$ be a \CM \ $A$-module of dimension $2$ and let $\F = \{ \F_n \}_{n \geq 0}$ be an $I$-stable filtration on $N$. If $e_2^\F(N) = 0$ then $H^1(L^\F(N)) = 0$.
\end{lemma}

\s\label{dim1}
Let $(A,\m)$ be a local ring with infinite residue field and let $I$ be an $\m$-primary ideal in $A$.
Let $N$ be \CM \ $A$-module of dimension one and let $\F = \{ \F_n \}_{n \geq 0}$ be an $I$-stable filtration on $N$. Let $x$ be $N$-superficial with respect to $\F$. Note $\F_{n+1} = x\F_n$ for all $n \gg 0$. Set $\rho^\F(z) = \sum_{n \geq 0} \ell(\F_{n+1}/x\F_n) z^n$.
\begin{proposition}\label{dim1-prop}
(with hypotheses as in \ref{dim1}) Then
\begin{enumerate}[\rm (1)]
  \item $H(\F, n) = e^\F(N) - \ell(\F_{n+1}/x\F_n)$ for all $n \geq 0$.
  \item $h^\F_N(z) = e^\F(N) + (z-1)\rho^\F(z)$.
  \item $e_1^\F(N) = \sum_{n \geq 0}\ell(\F_{n+1}/x\F_n)$ and $e_2^\F(N) = \sum_{n \geq 1}n\ell(\F_{n+1}/x\F_n)$.
\end{enumerate}
\end{proposition}
\begin{proof}
  (1) Consider the chain $\F_n \supseteq \F_{n+1} \supseteq x\F_n$. This yields
  $$H(\F, n) = \ell(\F_n/x\F_n) - \ell(\F_{n+1}/x\F_n).$$
The chain $N \supseteq xN \supseteq x\F_n$ yields
$$ \ell(N/x\F_n) = e^\F(N) + \ell(xN/x\F_n) = e^\F(N) + \ell(N/\F_n).$$
Finally the chain  $N \supseteq \F_n \supseteq x \F_n$ yields
$$ \ell(N/xF_n) =  \ell(N/\F_n) + \ell(\F_n/x\F_n).$$
Comparing second and third equalities we get $\ell(\F_n/x\F_n) = e^\F(N)$. The result follows.

(2) folllows from (1) and (3) follows from (2).
\end{proof}

\s \label{dim2}
Let $(A,\m)$ be a local ring with infinite residue field and let $I$ be an $\m$-primary ideal in $A$.
Let $N$ be \CM \ $A$-module of dimension two and let $\F = \{ \F_n \}_{n \geq 0}$ be an $I$-stable filtration on $N$. Let $x, y$ be $N$-superficial sequence with respect to $\F$.  Set $J = (x,y)$. Note $\F_{n+1} = J\F_n$ for all $n \gg 0$. Let $\ov{\F}$ be the induced filtration on $\ov{N} = N/xN$. Then note that  we have an exact sequence
$$ \frac{\F_{n+1} \colon x}{\F_n} \xrightarrow{f} \frac{\F_{n+1}}{J \F_n} \xrightarrow{\pi} \frac{\ov{\F}_{n+1}}{y \ov{\F_n}} \rt 0,$$
where $\pi$ is the obvious map and $f(p + \F_n) = xp + J\F_n$.
\begin{proposition}\label{dim2-prop}
(with hypotheses as in \ref{dim2})
Assume $\depth G_\F(N) \geq 1$. Then
$$e_1^\F(N) = \sum_{n \geq 0}\ell(\F_{n+1}/J\F_n) \quad \text{ and}  \quad e_2^\F(N) = \sum_{n \geq 1}n\ell(\F_{n+1}/J\F_n).$$
\end{proposition}
\begin{proof}
As $\depth G_\F(N) > 0$ we have
$$e_1^\F(N) = e_1^{\ov{\F}}(\ov{N}) \quad \text{and} \quad e_2^\F(N) = e_2^{\ov{\F}}(\ov{N}).$$
We also have $(\F_{n+1} \colon x) = \F_n$ for all $n \geq 0$. So by above exact sequence we get
$$   \ell\left(\frac{\F_{n+1}}{J \F_n}\right)  = \ell\left(\frac{\ov{\F}_{n+1}}{y \ov{\F_n}}\right) \quad \text{for all $n \geq 0$}. $$
The result now follows from \ref{dim1-prop}.
\end{proof}
We now give
\begin{proof}[Proof of Theorem \ref{e2}]
We may assume that the residue field of $A$ is infinite.
Let $\G$ be the Ratliff-Rush filtration of $\F$. Then note that $e_2^\G(N) = e_2^\F(N)$ (see \ref{basic-rr}(3)) and $H^1(L^\G(N)) = H^1(L^\F(N))$, see \ref{l-rr}. So we may replace $\F$ by $\G$. By
\ref{basic-rr}(4) we have $\depth G_\G(N) > 0$. So by \ref{dim2-prop}  we have
\[
0 = e_2^\G(N) = \sum_{n \geq 1}n\ell(\G_{n+1}/J\G_n).
\]
It follows that $\G_{n+1} = J\G_n$ for all $n \geq 1$. So
$\G_{n+1}\cap J N = J\G_n$ for all $n \geq 0$. Therefore $G_\G(N)$ is \CM \ by Valabrega-Valla Theorem. So $H^1(L^\G(N)) = 0$, by \ref{l-induct}. The result follows.
\end{proof}
\section{Some preliminaries on generalized Narita ideals}
In this section we discuss some preliminaries on generalized  Narita ideals that we need.

\s\label{prev-narita} Let $M$ be a \CM \ $A$-module of dimension $r \geq 2$ and let $I$ be an $\m$-primary ideal. In \cite[6.2]{Pu6} we proved that $e_2^I(M) = \cdots = e_r^I(M) = 0$ if and only if $\wt{G}_I(M)$, the associated graded module of the Ratliff-Rush filtration of $M$ \wrt \ $I$,
is \CM \ with minimal multiplicity, i.e., $\deg \wt{h}_{I, M}(z) \leq 1$.
Furthermore we have $H^i(L^I(M)) = 0$ for $1 \leq i \leq r -1$.

It is convenient to prove the following:
\begin{theorem}\label{main-body}
Let $(A,\m)$ be a \CM \ local ring of dimension $d \geq 2$. Let $I$ be an $\m$-primary generalized  Narita ideal. Let  $M$ be a MCM $A$-module. Then
 \begin{enumerate}[\rm (1)]
   \item $e_i^I(M) = 0$ for $2 \leq i \leq d$.
   \item $H^i(L^I(M)) = 0$ for $1\leq i \leq d-1$.
   \item $\wt{G}_I(M)$ is \CM \ with minimal multiplicity.
   \item $G_{I^n}(M)$ is \CM \ for all $n \gg 0$.
   \item $G_I(M)$ is generalized \CM.
 \item Let $H = A^r \rt M \rt 0$ be a projective cover. Then the map $H^0(L^I(H)) \rt H^0(L^I(M))$ is surjective.
 \end{enumerate}
 \end{theorem}
\section{Proof of Theorem \ref{main-body} when $d = 2$}
\begin{proof}
We may assume that the residue field of $A$ is infinite,

 (1) As $e_2^I(A) = 0$, by Narita there exists $n_0$ such that for $n \geq n_0$ then ideal $I^n$ has reduction number one. As $M$ is MCM $A$-module it follows that $M$ has minimal multiplicity with respect to $I^n$ for $ n \geq n_0$. So $e_2^{I^n}(M) = 0$. But $e_2^I(M) = e_2^{I^n}(M)$. It follows that $e_2^I(M) = 0$.

 (2), (3) and (4) By \cite[3.5, 6.2]{Pu6},
  $H^1(L^I(M)) = 0$ and $\wt{G}_I(M)$ is \CM \ with minimal multiplicity. It follows that $G_{I^n}(M)$ is \CM  \ for $n \gg 0$.

 (4) As $G_{I^n}(M)$
is \CM \ for all $n \gg 0$ it follows from \cite[7.8]{Pu5}
  $G_I(M)$ is generalized \CM.

 (5) Let $N = \Syz^A_1(M)$. The filtration $\F = \{\F_n = I^nH \cap N \}_{n \geq 0}$ is an $I$-stable filtration on $N$. We have an exact sequence
 $$ 0 \rt L^\F(N) \rt L^I(H) \rt L^I(M) \rt 0.$$
 So $e_2^\F(N) = 0$. It follows from \ref{e2} that $H^1(L^\F(N)) = 0$. So the map $H^0(L^I(H)) \rt H^0(L^I(M))$ is surjective.
\end{proof}
\section{A Lemma when $d \geq 3$}
We need the following result to prove our main theorem.
\begin{lemma}\label{rachel}
Let $(A,\m)$ be a \CM \ local ring and let $N$ be a \CM \ $A$-module of dimension $r \geq 3$.  Let $I$ be an $\m$-primary ideal in $A$ and  let $\F = \{ \F_n \}_{n \geq 0}$ be a
Assume $\F$ is an $I$-stable filtration on $N$.
\begin{enumerate}[\rm (1)]
  \item $e_i^\F(N) = 0$ for $i = 2, \ldots, r$.
  \item $H^i(L^\F(N))$ has finite length for $i = 1, \ldots, r -1$.
\end{enumerate}
Then $H^i(L^\F(N)) = 0$  for $i = 1, \ldots, r -1$.
\end{lemma}
\begin{proof}
We may assume the residue field of $A$ is infinite. It suffices to prove the result for
 $\G$  the Ratliff-Rush filtration with respect to $\F$ on $N$.
  We prove the result by induction on $r$.
  We first consider the case $r = 3$.  Let $x$ be $\G$-superficial and let $\ov{\G}$ be the quotient filtration on $\ov{N} = N/xN$. We note that $e_2^\G(N) = e_3^\G(N) = 0$. So $e_2^{\ov{\G}}(\ov{N}) = 0$. The exact sequence
  $$0 \rt L^\G(N)(-1)\xrightarrow{xt} L^\G(N) \rt L^{\ov{\G}}(\ov{N}) \rt 0, $$
  induces a long exact sequence in cohomology
 % $$  \rt H^1(L^\G(N))(-1) \rt H^1(L^\G(N)) \rt $$
 \begin{align*}
   &\rt H^1(L^\G(N))(-1) \rt H^1(L^\G(N)) \rt  H^1(L^{\ov{\G}}(\ov{N})) \\
   &\rt H^2(L^\G(N))(-1) \rt H^2(L^\G(N))
    \end{align*}

As $e_2^{\ov{\G}}(\ov{N}) = 0$ it follows from \ref{e2} that $ H^1(L^{\ov{\G}}(\ov{N}))  = 0$. As $H^i(L^\G(N))$ has finite length for $i = 1, 2$ we get   $H^i(L^\G(N))(-1) \cong H^i(L^\G(N))$.
This implies $H^i(L^\G(N)) = 0$ for $i = 1, 2$. The result follows.

We now assume that $r \geq 4$ and the result holds for $r-1$. Let $\G$ be the Ratliff-Rush filtration with respect to $\F$ on $N$. Let $x$ be $\G$-superficial and let $\ov{\G}$ be the quotient filtration on $\ov{N} = N/xN$. We note that $e_i^\G(N)  = 0$ for $2\leq i \leq r$. So $e_i^{\ov{\G}}(\ov{N}) = 0$ for $2\leq i \leq r -1$. The exact sequence
  $$0 \rt L^\G(N)(-1)\xrightarrow{xt} L^\G(N) \rt L^{\ov{\G}}(\ov{N}) \rt 0, $$
  induces a long exact sequence in cohomology for $i \geq 1$
 % $$  \rt H^1(L^\G(N))(-1) \rt H^1(L^\G(N)) \rt $$
 \begin{align*}
   &\rt H^i(L^\G(N))(-1) \rt H^i(L^\G(N)) \rt   H^i(L^{\ov{\G}}(\ov{N}))\\
   &\rt H^{i+1}(L^\G(N))(-1) \rt H^{i+1}(L^\G(N))
    \end{align*}
    We note that  $ H^1(L^{\ov{\G}}(\ov{N}))  $ has finite length for $i = 1, \ldots, r-2$. So by induction hypothesis  $H^i(L^{\ov{\G}}(\ov{N})) = 0$ for $i = 1, \ldots, r-2$.
    As $H^i(L^\G(N))$ has finite length for $i = 1, \ldots, r-1$ we get   $H^i(L^\G(N))(-1) \cong H^i(L^\G(N))$.
This implies $H^i(L^\G(N)) = 0$ for $1\leq i \leq r-1$ . The result follows.
\end{proof}
\section{Proof of Theorem \ref{main}}
In this section we give a proof of Theorem \ref{main-body}   by induction on $d = \dim A$. This will prove Theorem \ref{main}. We have already proved in the case when $d = 2$. In this section we first prove when $d = 3$ and then prove by induction the case when $d \geq 4$.
\begin{proof}
    Let $p \colon H = A^r \rt M$ be a projective cover. Let $N = \Syz^A_1(M) = \ker p$. The filtration $\F = \{\F_n = I^nH \cap N \}_{n \geq 0}$ is an $I$-stable filtration on $N$. We have an exact sequence
    \begin{equation*}
       0 \rt L^\F(N) \rt L^I(H) \rt L^I(M) \rt 0. \tag{$\dagger$}
    \end{equation*}

 We first consider the case when $d = 3$.
 So $e_2^\F(N) = e_2^I(M) = 0$. Let $x \in I $ be superficial with respect to the $I$-adic filtration on $H$ and $M$ and the filtration $\F$ on $N$. We note that
 as $e_2^I(\ov{M}) = e_2^{\ov{\F}}(N) = 0$ we have $H^1(L^I(\ov{M})) = H^1(L^{\ov{\F}}(\ov{N}) = 0$, by \ref{e2}. It follows that $H^2(L^I(M)) = H^2(L^\F(N)) = 0$, see \ref{l-induct}.
 By \ref{prev-narita}, we also have $H^1(L^I(H)) = H^2(L^I(H)) =0$. By the short exact sequence above we get $H^1(L^I(M)) = 0.$
 Thus the associated graded module of the Ratliff-Rush filtration on $M$ is \CM. Also note we have a surjection $G_I(H) \rt G_I(M)$. It follows that $a_3(G_I(M)) \leq a_3(G_I(H)) = -2$. Note
 $a_3(\wt{G}_I(M)) = a_3(G_I(M)) \leq -2$. It follows that $\wt{G}_I(M)$ has minimal multiplicity. So in particular $e_3^I(M) = 0$.
As  $\wt{G}_I(M)$ is \CM \  we have $G_{I^n}(M)$ is \CM \ for all $n \gg 0$.
 By \cite[7.8]{Pu5} we also get that $G_I(M)$ is generalized \CM.
 Note $e_3^{\F}(N) = 0$. By the above exact sequence as $H^1(L^I(H)) = 0$ it follows that $H^1(L^\F(N)) $ has finite length. Also $H^2(L^{\F}(N)) = 0$. By Lemma \ref{rachel} it follows that
 $H^1(L^\F(N)) = 0$. Thus the map $H^0(L^I(H)) \rt H^0(L^I(M))$ is surjective.

 We now assume $d \geq 4$ and the result is known for $d -1$.
 Let $x \in I $ be superficial with respect to the $I$-adic filtration on $H$ and $M$ and the filtration $\F$ on $N$.
 We have $e_i^I(\ov{A}) = 0$ for $i = 2,\ldots, d - 1$. By induction hypotheses we have $e_i^I(\ov{M}) = 0$ for $i = 2, \ldots, d -1$ and $H^i(L^I(\ov{M})) = 0$ for $i = 1, \ldots, d -2$. Therefore
 $H^i(L^I(M)) = 0$ for $2 \leq i \leq d -1$, see \ref{l-induct}. We also have $H^i(L^I(H))  = 0$ for $1 \leq i \leq d -1$
 We have an  exact sequence
 \begin{equation*}
       0 \rt C \rt L^{\ov{\F}}(\ov{N}) \rt L^I(\ov{H}) \rt L^I(\ov{M}) \rt 0, \tag{$*$}
    \end{equation*}
    where $C$ has finite length. As $e_i^I(\ov{H}) = e_i^I(\ov{M}) = 0$ for $2 \leq i \leq d -1$ it follows that $e_i^{\ov{F}}(N) = 0$ for $2 \leq i \leq d -1$.
    Also by (*) it follows that $H^i(L^{\ov{\F}}(N)) = 0$ for $i = 2, \ldots d-2$ and has finite length when $i = 1$. By \ref{rachel} it follows that $H^i(L^{\ov{\F}}(N)) = 0$ for $i = 1, \ldots d-2$.
    Thus $H^i(L^\F(N)) = 0$ for $i =  2, \ldots, d -1$, see \ref{l-induct}. By $(\dagger)$ it follows that $H^1(L^I(M)) = 0$.  Thus the associated graded module of the Ratliff-Rush filtration on $M$ is \CM. Also note we have a surjection $G_I(H) \rt G_I(M)$. It follows that $a_d(G_I(M)) \leq a_d(G_I(H)) = -d +1$. Note
 $a_d(\wt{G}_I(M)) = a_d(G_I(M)) \leq -d +1$. It follows that $\wt{G}_I(M)$ has minimal multiplicity. So in particular $e_d^I(M) = 0$.
 As $\wt{G}_I(M)$  is \CM \ we have $G_{I^n}(M)$ is \CM \ for $n \gg 0$.
 By \cite[7.8]{Pu5} we also get that $G_I(M)$ is generalized \CM.
 From the exact sequence $(\dagger)$ we get that $e_i(L^\F(N)) = 0$ for $2 \leq i \leq d$. We also have  $H^i(L^\F(N)) = 0$ for $i =  2, \ldots, d -1$ and $H^1(L^\F(N))$ is a quotient of $H^0(L^I(M))$ and so has finite length. By Lemma \ref{rachel},   we get $H^1(L^\F(N)) = 0$ also. Thus the map $H^0(L^I(H)) \rt H^0(L^I(M))$ is surjective.
\end{proof}

\section{Regularity of $G_I(M)$}
In this section we prove Theorem \ref{reg}.
\begin{proof}
If $\depth G_I(A) > 0$ then $I$ has minimal multiplicity. It follows that $G_I(M)$ is also \CM \ for any MCM module. So $\reg G_I(M) \leq 1$.

Next assume $\depth G_I(A) = 0$. So $\e H^0(L^I(A)) \geq 0$.

  First assume $\depth G_I(M) > 0$. Then $G_I(M) = \wt{G}_I(M)$ has minimal multiplicity. In particular $\reg G_I(M) \leq 1$.
  Next assume that $\depth G_I(M) = 0$. Then as $H^i(L^I(M)) = 0$ for $i = 1, \ldots, d -1$ we get that $H^i(G_I(M)) = 0$ for $i = 2, \ldots, d -1$ and we have a sequence
  $$ 0 \rt H^0(G_I(M)) \rt H^0(L^I(M)) \rt H^0(L^I(M))(-1) \rt H^1(G_I(M)) \rt 0. $$
  So we obtain
  $$a_0(G_I(M)) < a_1(G_I(M)) = \e H^0(L^I(M)) + 1. $$
  As $\wt{G}_I(M)$ has minimal multiplicity we have $a_d(G_I(M)) = a_d(\wt{G}_I(M)) \leq  -d + 1$.
  So we have $\reg G_I(M) \leq \e  H^0(L^I(M)) + 2$. Let $p \colon H = A^r \rt M \rt 0$ be a minimal map. By \ref{main-body} we have a surjection $H^0(L^I(H)) \rt H^0(L^I(M))$. So $\e  H^0(L^I(M)) \leq  \e  H^0(L^I(H )) = \e  H^0(L^I(A)).$  So $\reg G_I(A) \leq \e H^0(L^I(A)) + 2$. Set $c_I = \e H^0(L^I(A)) + 2$. The result follows.
\end{proof}

\section{Example of generalized Narita ideals}
In this section we prove that every complete three dimensional \CM \ local ring containing a field of characteristic zero has a generalized Narita ideal.
The following example is taken from \cite[3.8]{Cpr}
 \begin{example}\label{CPR}
 Let $A = \mathbb{Q}[[x,y,z]]$. Let $I= (x^2-y^2, y^2-z^2, xy,xz,yz)$. Set $\m = (x,y,z)$. It can be verified that
 $I^2 =  \m^4$. So $\red(I^n) = 2$ for all $n \gg 0$.
 Using COCOA  we get
 $h_I(A,z) =  5 + 6z^2 - 4z^3+z^4$. So $e_j^I(A) = 0$ for $j = 2,3$.
 In \cite{Cpr} it is  proved that $\depth G_I(A) = 0$.
\end{example}

Let $K$ be a field of characteristic zero. Set $B = K[[x, y, z]]$.
We note that $A \rt B$ is a flat extension with fiber $K$. It follows that $J = IB$ is a generalized Narita ideal in $B$.

\begin{example}
Let $(R, \m)$  be a complete three dimensional \CM \ local ring containing a field of characteristic zero. Set $K = R/\m$. Then $R$ contains a subring $B = K[[x, y, z]]$ such that
$R$ is finite $B$-module. We note that as $R$ is \CM \ it follows that $R$ is a free $B$-module. By above assertion $B$ has a generalized  Narita ideal $J$. Then it is easy to verify that
$JR$ is a generalized Narita ideal in $R$.
\end{example}

\section{Proof of Theorem \ref{almost}}
In this section we give a proof of Theorem \ref{almost}.
\s Let $I$ be an $\m$-primary ideal in $A$. Let $\wh{\R} = \bigoplus_{n \in \Z}I^nt^n$ be the extended Rees algebra of $I$ (here we set $I^n = A$ for $n \leq 0$) considered as a subring of $A[t, t^{-1}]$. If $M$ is an $A$-module then set $\wh{\R}(M) = \bigoplus_{n \in \Z}I^nM$ to be the extended Rees module of $M$. Then it can be shown that components of local cohomology modules $H^i_{\R_+}(\wh{\R}(M))_n$ has finite length for all $n \in \Z$ and is zero for $n \gg 0$. Let $r = \dim M$. We have $H^i_{\R_+}(\wh{\R}(M)) = 0$ for $i >  r$.  We also have the following
\begin{equation*}
  P^1_I(M, n) - H^1_I(M, n) = \sum_{i = 0}^{r}(-1)^i\ell(H^i_{\R_+}(\wh{\R}(M)_{n+1}). \tag{$**$}
\end{equation*}
Note that the index in the right hand side is $n+1$ and not $n$.
These results were proved in the case when $M = A$ in \cite[4.1]{B}. Same proof works in general.

\s Consider the exact sequence
\[
0 \rt \wh{\R}(M) \rt M[t, t^{-1}] \rt L^I(M)(-1) \rt 0.
\]
This yields $L^I(M)(-1)$ ( and hence $L^I(M)$ a structure of a $\wh{R}$-module. Furthermore our earlier structure of $L^I(M)$ as a $\R$-module  is by restriction of scalars along the inclusion mapping $\R \hookrightarrow \wh{\R}$.

 \s If $M$ is \CM \ of dimension $r \geq 1$ then note that there exists \\ $x_1, \ldots, x_r \in I$ which is $M$-regular. It follows that
for $i = 0, \ldots, r-2$ we have isomorphism $H^i_{\R_+}(L^I(M))(-1) \cong H^{i+1}_{\R_+}(\wh{\R}(M))$ and $H^{r-1}_{\R_+}(L^I(M))(-1)$ is a submodule of $H^{r}_{\R_+}(\wh{\R}(M))$.

\s Let $\M$ be the maximal graded ideal of $\R$.  In \cite[3.1]{Pu7} we proved that the natural map $H^i_\M(L^I(M)) \rt H^i_{\R_+}(L^I(M))$ is an isomorphism for all $i \geq 0$.
Also as $G_I(M)$ is finitely generated $\R$-module
(with $G_I(M)_n$ of finite length for all $n$) we also have  that the natural map $H^i_\M(G_I(M)) \rt H^i_{\R_+}(G_I(M))$ is an isomorphism for all $i \geq 0$

\emph{Throughout this section by $H^i(-)$ we mean $H^i_{\R_+}(-)$}.

The following result is well-known. We give a proof due to lack of a suitable reference.
\begin{lemma}\label{a-inv-rr-g}
Let $(A, \m)$ be local and let $I$ be an $\m$-primary ideal in $A$. Let $M$ be an $A$-module of dimension $r$. It $H^r(G_I(M))_n = 0$ for $n \geq c$ then $H^r(\wh{\R}(M))_n = 0$ for $n \geq c$.
\end{lemma}
\begin{proof}
  The exact sequence $0 \rt \wh{\R}(M)(+1) \xrightarrow{t^{-1}} \wh{\R}(M) \rt G_I(M) \rt 0$ induces a sequence in local cohomology
  \[
  H^r(\wh{\R}(M))_{n + 1} \rt H^r(\wh{\R}(M))_n \rt H^r(G_I(M))_n \rt 0.
  \]
  As $H^r(G_I(M))_n = 0$ for $n \geq c$ we get surjections $H^r(\wh{\R}(M))_{n + 1} \rt H^r(\wh{\R}(M))_n  \rt 0$ for $n \geq c$. As $H^r(\wh{\R}(M))_j = 0$ for $j \gg 0$ the result follows.
\end{proof}
We now give
\begin{proof}[Proof of Theorem \ref{almost}]
Let $x$ be $M\oplus A$-superficial \wrt \ $I$. Set  As $e_i^{\ov{I}}(\ov{A}) = 0$ for $2 \leq i \leq d -1$ it follows that $\ov{I}$ is a generalized Narita ideal in $\ov{A}$. So we have $e_i^{\ov{I}}(\ov{M}) = 0$. It follows that $H^i(L^I(\ov{M}) = 0$ for $1 \leq i \leq r -2$ and that $a_{d-1}(G_I(\ov{M}) \leq -d +2 < 0$. It follows that $H^i(L^I(M)) = 0$ for $2 \leq i \leq d-1$. Furthermore it is easily verified that $a_d(G_I(M)) < -1$. So \ref{a-inv-rr-g} we have $H^d(\wh{R}(M))_n = 0$ for $n \geq -1$

Computing  (**) for $n = -1$ we obtain
\[
(-1)^{d}e_d^I(M) = \sum_{i = 0}^{d}(-1)^i\ell(H^i(\wh{\R}(M)_{0}).
\]
We note that $H^0(\wh{\R}(M)) = 0$. We have for $i = 1,\ldots, d - 1$ isomorphisms \\ $H^i(\wh{\R}(M))_0 = H^{i-1}(L^I(M))_{-1}$. So we obtain
$H^i(\wh{\R}(M))_0 = 0$ for $i = 1$ and $i = 3,\ldots, r - 1$. We also have by $H^d(\wh{R}(M))_0 = 0$. Thus we have
\[
(-1)^{d}e_d^I(M) = \ell(H^1(L^I(M))_{-1})
\]
Thus $(-1)^{d}e_d^I(M) \geq 0$.

(II) If $e_I^I(M) \neq 0$ then $H^1(L^I(M))_{-1} \neq 0$. It follows from \cite[9.2]{Pu5} that \\ $\depth G_{I^n}(M) = 1$ for all $n \gg 0$.

(I) If $e_d^I(M) = 0$ then it follows from \ref{prev-narita} that the associated graded module of the Ratliff-Rush filtration on $M$ \wrt \ $I$ is \CM \ with  minimal multiplicity. So $G_{I^n}(M)$ is \CM \ for $n \gg 0$ and this forces by \cite[7.8]{Pu5} that $G_I(M)$ is generalized  \CM.  Next assume that $G_I(M)$ is generalized \CM. Then by \cite[5.2]{Pu6} $H^1(L^I(M))$ has finite length.
By \ref{longH} we have a sequence
\[
H^1(L^I(M))(-1) \rt H^1(L^I(M)) \rt H^1(L^I(\ov{M})) = 0.
\]
It follows that $H^1(L^I(M)) \cong H^1(L^I(M))(-1)$ and so $H^1(L^I(M)) = 0$. Thus the associated graded module of the Ratliff-Rush filtration on $M$ \wrt \ $I$ is \CM. As $e_2^I(M) = 0$ it follows that $\wt{G}_I(M)$ has minimal multiplicity. So $e_d^I(M) = 0$.
\end{proof}

\end{document}